\documentclass[12pt,a4paper]{article}
\usepackage{ae,aecompl}
\usepackage{amssymb,amsmath,amsthm}
\usepackage[bookmarksnumbered=true,pdfstartview=FitH]{hyperref}
\usepackage[english]{babel}
\usepackage{fancyhdr}
\usepackage[top=3cm,right=2.5cm,left=2.5cm,bottom=2.3cm]{geometry}

\usepackage{xcolor}
\numberwithin{equation}{section}
\allowdisplaybreaks[4]
\newtheorem{thm}{Theorem}[section]

\newtheorem{lemma}[thm]{Lemma}
\newtheorem{prop}[thm]{Proposition}
\newtheorem{defin}[thm]{Definition}
\newtheorem{cor}[thm]{Corollary}
\theoremstyle{remark}
\newtheorem{rmk}[thm]{Remark}

\theoremstyle{remark}
\def\ts{\otimes}
\def\oh{\frac{1}{2}}

\def\ket#1{| #1 \rangle}
\def\te{\theta}
\newcommand\CA{{\mathcal A}}
\newcommand\CH{{\mathcal H}}

\def\oh{\frac{1}{2}}

\def\ket#1{| #1 \rangle}
\newcommand{\bR}{\mathbb{R}}
\newcommand{\Tb}{\mathbb{T}}         
\newcommand{\bt}{\textbf}

\newcommand{\cop}{\Delta}

\newcommand{\Z}{\mathbb{Z}}
\newcommand{\R}{\mathbb{R}}
\newcommand{\C}{\mathbb{C}}

\newcommand{\T}{\mathbb{T}}
\newcommand{\A}{\mathcal{A}}

\newcommand{\B}{\mathcal{B}}

\newcommand{\eps}{\varepsilon}

\renewcommand{\H}{\mathcal{H}}

\newcommand{\Cl}{\mathbb{C}\mathrm{l}}
\newcommand{\id}{\mathrm{id}}

\title{\vspace{-1cm}Dirac operators\\
on noncommutative principal circle bundles}
\date{}
\author{Ludwik D\k{a}browski${}^{1,*}$, Andrzej Sitarz${}^{2,3,**}$ and Alessandro Zucca$^{\,1}$\\ 
\vbox{
\small
\begin{center}
${}^1$ SISSA (Scuola Internazionale Superiore di Studi Avanzati), \\
via Bonomea 265, 34136 Trieste, Italy \\ \ \\
${}^2$ Institute of Physics, Jagiellonian University,\\
Reymonta 4, Cracow, 30-059 Poland \\ \ \\
${}^3$ Institute of Mathematics of the
Polish Academy of Sciences,\\
\'Sniadeckich 8, Warsaw, 00-950 Poland \\ \ \\
${}^{*}$ Partially supported by PRIN 2010-11 "Operator Algebras, Noncommutative Geometry and Applications".\\ \ \\
${}^{**}$ Partially supported by NCN grant 2011/01/B/ST1/06474.\\
\end{center}
}
}
\begin{document}
\maketitle
\begin{abstract}
We study spectral triples over noncommutative principal $U(1)$-bundles of arbitrary dimension
and a compatibility condition between the connection and the Dirac operator on the total 
space and on the base space of the bundle. Examples of low dimensional  noncommutative tori 
are analysed in more detail and all connections found that are compatible with an admissible 
Dirac operator. Conversely, a family of new Dirac operators on the noncommutative tori, 
which arise from the base-space Dirac operator and a suitable connection is exhibited. 
These examples are extended to the theta-deformed principal $U(1)$-bundle $S^3_\theta \to S^2$.
\end{abstract}

\thispagestyle{empty}
\section{Introduction}

The spectral triples \cite{C94,C96} and in particular the Dirac operators on noncommutative
odd-dimensional principal $U(1)$-bundles have been studied in \cite{DS10}.
Therein the definition of a connection as an operator and a compatibility condition with the Dirac operators on the total space and on the base space of the bundle were proposed.
This was based on the classical situation \cite{A98,AB98} and the abstract algebraic approach
to noncommutative principal bundles. 
Moreover, the case of the noncommutative 3-torus viewed as a $U(1)$-bundle 
over the noncommutative 2-torus was analysed and all connections compatible 
with an admissible Dirac operator were found. 
Conversely, a family of new Dirac operators on the odd-dimensional tori was constructed, 
that are consistent with the base-space (even) Dirac operator and a suitable connection.

In this paper by making an ampler use of the real structure we slightly modify the 
scheme of \cite{DS10} and so weaken the assumptions made there. In addition we extend 
the study to principal $U(1)$-bundles of arbitrary dimension. As examples we discuss two
and four dimensional  noncommutative tori and the theta deformation of the principal Hopf $U(1)$-bundle $S^3 \to S^2$, which was studied in \cite{BrzSi}.

\section{Quantum principal $U(1)$-bundles}

In the noncommutative realm the notion of quantum principal bundles have been recently 
formulated in the context of principal comodule algebras. This evolved from earlier extensive studies
of quantum principal bundles with universal differential calculus, of principal extensions, and 
of Hopf-Galois extensions, see e.g. \cite{BM93,DGH,HM99,H96}. 
For our purpose we shall also need a notion of the principal connection adapted to the 
case of Dirac operator induced calculi.


Let $H$ be a unital Hopf algebra over $\C$ with a counit 
$\varepsilon$ and an invertible antipode $S$. Let $\A$ be a right $H$-comodule algebra (which 
we also assume to be unital). We will use the Sweedler notation for the right 
coaction of $H$ on $\A$:
$$ \Delta_R(a) = a_{(0)} \otimes a_{(1)} \in \A \otimes H.$$ 
(We shall often skip the adjective right if it is clear from the context).
Then $\A$ is called {\em principal $H$-comodule algebra} (c.f. \cite{BZ11}) whenever 
there exists a $\C$-linear unital map 
$$\ell:H\longrightarrow \A\otimes \A$$ such that
\begin{subequations}
\label{strong}
\begin{gather}
m_\A\circ \ell = \eta \circ \varepsilon, \label{strong2}\\
 (\ell\otimes\id)\circ\Delta  = (\id\otimes \Delta_R)\circ \ell , \label{strong3}\\
(S\otimes \ell)\circ\Delta = (\sigma\otimes \id)\circ (\Delta_R\otimes \id)\circ \ell ,  \label{strong4}
\end{gather}
\end{subequations}
where $m_\A:\A\otimes \A\to \A$ is the multiplication in $\A$,  
$\eta:\C\to \A$ is the unit map and $\sigma:\A\otimes H\to H\otimes \A$ is the flip.

It follows then that $\A$ is a {\em Hopf-Galois extension} of 
$\B:=\{b\in \A\  \ |\ \Delta_R(a) = b \otimes 1  \}$, its subalgebra of coinvariant elements.
Namely 
$$
\A\otimes H\ni a\otimes h\longmapsto a\,\pi_B (\ell(h))\in \A\otimes_B\A ,
$$
where $\pi_B :\A\otimes\A \to \A\otimes_B\A $ is the usual projection,
is the (left and right) inverse of the  {\em canonical map} 
$\chi:= (m_\A \ts \id)(\id \ts \cop_R)$,
\begin{equation}
\A \otimes_\B \A \ni  a' \otimes a \mapsto
\chi( a' \otimes a) = a' a_{(0)} \otimes a_{(1)} \in \A \otimes H.
\end{equation}

By quantum principal bundle we shall understand the structure defined as above,
and we shall often denote it simply by the inclusion map $\B \hookrightarrow \A$. 
Moreover, we shall call the map $\ell$ {\em strong connection lift}.
It determines the strong connection (for the universal calculus on $\A$)
which can be equivalently described in four other terms (cf. \cite{DGH}).
One of them is by a certain splitting of the multiplication map $\B\otimes\A \to \A$ 
(which shows the equivariant left projectivity of $\A$ over $\B$). 
Another one, the so called {\em connection form}, 
will be adapted in section \ref{omega} to the case of the Dirac-operator induced 
differential calculus, and represented as an operator on a Hilbert space.

In this paper we restrict our discussion only to the case of the structure group 
given by the classical one-dimensional Lie group $U(1)$. We shall denote by 
$H$ its coordinate algebra 
i.e. the complex polynomial algebra in $z$ and $z^{-1}$, with the star structure given by $zz^*=1$,
the coproduct by $\Delta(z)=z\otimes z$,  
a counit $\varepsilon(z)=1$ and an invertible antipode $S(z)=z^{-1}$.
We can split the $H$-comodule algebra $\A$ as a direct sum
\begin{equation}\label{AsumAk}
 \A = \bigoplus_{k\in\Z^n} \A^{(k)},
\end{equation}
where the $\A^{(k)}$ are the set of elements of homogeneous degree $k$; that is,
$$a\in \A^{(k)}\quad\Leftrightarrow \quad \Delta_R(a) = a\otimes z^k.$$
Of course, $\A^{(0)} = \B$.

\section{Dirac operator and spectral triples}

According to Alain Connes the noncommutative metric and spin  geometry is encoded in terms 
of spectral triples \cite{C94,C95,C96}. 
In what follows $\A$ will always denote a unital complex $^*$-algebra.

\begin{defin}\label{defsptr}
A \emph{spectral triple} for an algebra $\A$ is a triple $(\A,\H,D)$, where $\H$ is a Hilbert space carrying a representation of $\A$ by bounded operators (which we shall simply denote by $\psi\mapsto a\psi$, for any $a\in \A$, $\psi\in\H$) and $D$ is a densely defined, selfadjoint operator on $\H$ with compact resolvent, such that for any $a\in\A$ the commutator $[D,a]$ is a bounded operator.
\end{defin}

In the commutative case when $\A$ is the algebra of smooth functions over a Riemannian spin manifold $M$, $\H$ corresponds to the Hilbert space of $L^2$-sections of the spinor bundle, and $D$ is the Dirac operator, on the spinor bundle, associated to the Levi-Civita connection. So, given a spectral triple $(\A,\H,D)$, we will call $\H$ the \emph{space of spinors} and $D$ the \emph{Dirac operator}.
The algebra $\A$ will be always certain suitable *-algebra of operators on $\H$, 
dense in its norm completion (C*-algebra), in particular it can be a (noncommutative)
polynomial algebra presented in terms of generators and relations. 

An important requirement on spectral triples is that of reality.
\begin{defin}\label{defrsptr}
A \emph{real spectral triple} of \emph{$KR$-dimension} j, where $j\in\Z_8$,
 consists of the data
$(\A,\H,D,J)$ when $j$ is odd, whereas of the data  $(\A,\H,D,J,\gamma)$ 
when $j$ is even, 
where $(\A,\H,D)$ is a spectral triple, $J$ is antiunitary operator and $\gamma$ is a $\Z_2$-grading 
on $\H$ such that:
\begin{description}
\item[(i)] for any $a,b\in\A$, $[a,Jb^*J^{-1}] = 0$;
\item[(ii)] $J$, $D$ and $\gamma$ satisfy the following commutation relations:
$$J^2 = \eps \id,\quad\quad JD = \eps' DJ$$
and, for $j$ even,
$$J\gamma = \eps''\gamma J,\quad\quad \gamma D = - D\gamma,$$
where $\eps,\eps',\eps''$ depend on the $KR$-dimension and are given in the following table 
(c.f. \cite{DD10}, which admits additional possibilities with respect to the table in 
\cite{C96}, whose entries are marked by $\bullet$)

\begin{table}[!ht]\label{ko-dim}
\centering
\begin{tabular}{|c||c|c|c|c||c|c|c|c||c|c|c|c|}
\hline
$j$	&	0	&	2	&	4	&	6	&	0	&	2	&	4	&	6 &	1	&	3	&	5	&	7\\
\hline
$\epsilon$	&	$+$	&	$-$  &	$-$	&	$+$	&	$+$	&	$+$	&	$-$	&	$-$	&	$+$	&	$-$	&	$-$ &	$+$\\
$\epsilon'$	&	$+$	&	$+$	&	$+$	&	$+$	&	$-$	&	$-$	&	$-$	&	$-$	&	$-$	&	$+$	&	$-$	&	$+$\\
$\epsilon''$	&	$+$	&	$-$	&	$+$	&	$-$	&	$+$	&	$-$	&	$+$	&	$-$	&	&	&	&\\
\hline
&	$\bullet$	&$\bullet$&$\bullet$&$\bullet$&	&	&	&	&$\bullet$&$\bullet$&$\bullet$&$\bullet$\\
\hline
\end{tabular}
\end{table}
\end{description}
\end{defin}
The operator $J$ will be usually called the \emph{real structure} of the spectral triple. 
We will often treat together the even and the odd dimensional case. 
Thus in general we will write $(\A,\H,D,J,\gamma)$ for a real spectral triple, keeping in mind
that in the odd case $\gamma=1$ and so will be often omitted.

The antiunitary operator $J$ determines a left action of the opposite algebra $\A^\circ$ 
(or, equivalently, a right action of the algebra $\A$) on the Hilbert space $\H$,
\begin{equation}\label{JaJ}
\pi^\circ(b)\psi = \psi b = Jb^*J^{-1}\psi, \quad \forall\, b\in\A, \psi\in\H,
\end{equation}
which due to the condition (i) of definition \ref{defrsptr} 
commutes with the representation of $\A$ on $\H$; that is, $J$ maps $\A$ into its commutant on $\H$. 
This makes the Hilbert space $\H$ a bi-module over $\A$. 

Moreover we shall assume
the so-called \emph{first order condition}, that is the requirement that $\A^\circ$ 
commutes not only with $\A$ but also with $[D,\A]$ so that, for any $a,b\in\A$, 
we have:
\begin{equation}\label{firstordcond}
[[D,a],Jb^*J^{-1}] = 0.
\end{equation}
Notice that conjugating the above identity with $J$ one can show \eqref{firstordcond} 
to be equivalent to
$$[[D,Jb^*J^{-1}],a] = 0;$$
and so the first order condition is ``symmetric'' in $\A$ and $\A^\circ$.

\section{$U(1)$--equivariant projectable spectral triples}

We will assume that the coaction of $H$ on $\A$  (which in the context of $C^*$-algebras 
is used to formulate the notion of free and proper action) comes from the usual action 
of $U(1)$ on $\A$. Moreover, we will require that there is an infinitesimal action $\delta$ 
of the generator of the Lie algebra $u(1)\simeq\R$.

To use more precise language, we shall consider real spectral triples over $\A$, which 
are $U(1)$-equivariant \cite{Sitarz03}, that is, for which the action of $U(1)$ (and of $u(1)$ 
via $\delta$) is implemented on $\H$ as follows.
\begin{defin}\label{defU1equivst}
A  real spectral triple $(\A,\H,D,J,\gamma)$ is a \emph{$U(1)$-equivariant real spectral
triple}
iff there exists a selfadjoint operator $\hat\delta$ on $\H$ with the domain which is dense 
and stable under the action of $\A$, $D$, $J$ and $\gamma$,  such that 
$$\hat\delta(\pi(a)\psi) = \pi(\delta(a))\psi + \pi(a) \hat\delta(\psi),$$
and 
$$\hat\delta J + J\hat\delta = 0,\quad\quad [\hat\delta,\gamma] = 0,
\quad\quad [\hat\delta,D] = 0, $$
on the common core of $\hat\delta$ and $D$. 
\end{defin}

Actually, we require also that the spectrum of $\hat\delta$ is $\Z$
(it could be also $\Z + {1\over 2}$): 
this corresponds to the assumption that the $U(1)$ action on the tangent bundle lifts 
to an action and not to a projective action on the spinor bundle.
Hence, if $(\A,\H,D,J,\gamma,\hat\delta)$ is a $U(1)$-equivariant real spectral triple, 
we can split the Hilbert space $\H$ accordingly to the spectrum of $\hat\delta$:
$$\H = \bigoplus_{k\in\Z} \H_k,$$
and this decomposition is preserved by the Dirac operator $D$. Moreover 
$\pi(\A^{(k)})\H_l\subseteq \H_{k+l}$ for any $k,l\in\Z$; in particular $\H_0$ is stable 
under the action of the invariant subalgebra 
$\B = \A^{coH} = \A^{(0)}$.
Throughout the rest of the paper we shall denote $\hat\delta$ for simplicity by $\delta$.


\subsection{Operator of strong connection for the Dirac calculus}\label{omega}

With a spectral triple $(\A,\H,D)$ there is associated the (first order differential) 
Dirac calculus $\Omega^1_{D}(\A)$, given by the linear span of all operators of the form 
$a' [D, a]$,  $a, a' \in \A$, with the differential of $a$ given by $da = [D, a]$.
On the Hopf algebra $H$ we shall use the calculus $\Omega^1(H)$ induced from the
usual de Rham calculus on $U(1)$ (which can be also viewed as Dirac calculus 
for the operator $i\frac{\partial}{\partial \phi} $ on $L^2(U(1))$).


Following \cite{DS10} we say that $\Omega^1_{D}(\A)$ is {\em compatible} 
with $\Omega^1(H)$ iff
\begin{equation}
 \sum_i p_i [D, q_i] = 0, {\rm \, for \ }  p_i,q_i \in \A \Rightarrow 
\sum_i p_i \delta(q_i) = 0.
\label{eq43}
\end{equation}
This agrees with the compatibility of these calculi viewed as quotients of the universal 
calculi. 
Accordingly we can give the notion of connection form adapted to the Dirac-induced 
calculi, represented as operators on certain Hilbert space, as follows:
\begin{defin}
\label{defcon}
We say that $\omega \in \Omega^1_D(\A)$ is a strong principal connection 
for the $U(1)$-bundle $\B \hookrightarrow \A$ if the following conditions hold:
\begin{align*}
& [\delta, \omega] = 0, \;\; \hbox{\em ($U(1)$ invariance of $\omega$),} \\
& \hbox{if\ } \omega = \sum_i p_i [D, q_i]\; \hbox{then} \sum_i p_i \delta(q_i) = 1, \;\; 
\hbox{\em (vertical field condition)}, \\
& \forall a \in \A,  \nabla_\omega (a):=[D,a] - \delta(a) \omega \in \Omega_D^1(\B) \A, 
\;\; \hbox{\em (strongness).} 
\end{align*}
\end{defin}
Note that the second condition is meaningful due to assumption (\ref{eq43}). 

The expression in the third condition defines a $D$-connection $\nabla_\omega$ 
on the left $\B$-module $\A$ (covariant derivative with respect to the calculus 
$\Omega_D^1(\B)$), in the sense that
\begin{equation}
\nabla_\omega (ba)  = [D,b]a + b\nabla_\omega  (a).
\label{covderA}
\end{equation}
It follows by direct computation, see lemma 5.5 in \cite{DS10}, 
that a $D$-connection $\nabla_\omega$ is hermitian if $\omega$ is selfadjoint 
as an operator on $\H$.

\subsection{Projectable spectral triples: odd case}\label{ssectprojU1odd}
Now we discuss  in the framework of noncommutative geometry the projectability of a spectral triple, i.e. that it descends via a reduction by dimension one to the quotient space. 
We have to distinguish the case of  odd and of even dimension.
We begin by considering the former.

Let $\B\hookrightarrow\A$ be a principal $H$-comodule algebra, where $H$ is 
the coordinate algebra of $U(1)$, and consider a $U(1)$-equivariant odd real spectral 
triple $(\A,\H,D,J,\delta)$ over the total space $\A$. 
Assume that the Dirac calculus $\Omega^1_D(\A)$ is compatible with the de Rham 
calculus on $H$.
Then we give the following definition (c.f. \cite{DS10}).
\begin{defin}\label{defprojoddst}
An odd $U(1)$-equivariant real spectral triple $(\A,\H,D,J,\gamma,\delta)$ 
of $KR$-dimension $j$ 
is said to be \emph{projectable} along the fibers if there exists a $\Z_2$ grading 
$\Gamma$ on $\H$, which satisfies the following conditions,
$$\Gamma^2 = 1,\quad\quad \Gamma^* = \Gamma,$$
$$[\Gamma,\pi(a)]=0 \quad\forall a\in\A,$$
$$[\Gamma,\delta] = 0,$$
\begin{equation*}
\Gamma J = \left\{\begin{array}{lcl}
J\Gamma & \quad & \mathrm{if}\; j \equiv 1 \; (\mathrm{mod}\; 4) \\
-J\Gamma & \quad & \mathrm{if}\; j \equiv 3 \; (\mathrm{mod}\; 4).
\end{array}\right.
\end{equation*}
If such a $\Gamma$ exists, we define the \emph{horizontal Dirac operator} $D_h$ by:
$$D_h\equiv {1\over 2}\Gamma[\Gamma,D]$$
\end{defin}

Now we want to consider a special class of projectable spectral triples, which should represent 
the noncommutative counterpart of smooth $U(1)$ principal bundles which are Riemannian manifolds
 with fibers of constant length.

\begin{defin}\label{defU1constlength}
 A projectable spectral triple \emph{has fibers of constant length} if there is a positive real number 
$l$ such that, if we set
$$D_v = {1\over l}\Gamma\delta,$$
the operator
$$Z = D - D_h - D_v$$
is a bounded operator which commutes with the representation of $\A$
\begin{equation}
\label{za}
[Z, \A] = 0.
\end{equation}
\end{defin}
In such a case, the operator $D_v$ is called the \emph{vertical Dirac operator}
and the number $l$, up to a $2\pi$ factor is the length of each fiber,  as in the 
commutative (smooth) case.
\begin{rmk}
Actually in \cite{DS10} a different condition on $Z$ was imposed, namely that $Z$ 
commutes with the elements $J\pi(a^*)J^{-1}$ from the commutant of $\A$.
Our present choice is satisfied e.g if $Z \in J\pi(\A)J^{-1}$, 
or if $Z$ is a related 0-order pseudodifferential operator, 
and is motivated by a recent example of curved noncommutative torus \cite{DS12}. 
The condition \ref{za} implies that
\begin{equation}
\label{projcalc}
[D_h,\pi(b)] = [D,\pi(b)],  \quad  \forall b\in\B
\end{equation}
and hence $D_h$ and $D$ induce the same first order differential calculus 
over the invariant subalgebra $\B$ which was an additional assumption of 
'projectability of differential calculus' in \cite{DS10}. $\quad\diamond$
\end{rmk}
Consider now a projectable triple $(\A,\H,D,J,\delta,\Gamma)$ of $KR$-dimension $j$, 
with $j$ odd, and assume that it has fibers of constant length. Since $\Gamma$ 
and $D$ commute with $\delta$, also $D_h$ does. 
Therefore $D_h$ preserves each $\H_k$. Instead, the real structure intertwines
$\H_k$ and $\H_{-k}$, $J\H_k\subseteq \H_{-k}$. In particular, it preserves $\H_0$. 
Now, let us denote, for any $k\in\Z$, by $D_k$, and $\gamma_k$ the restrictions 
to $\H_k$ of $D_h$ and $\Gamma$, respectively. 
For each $k\in\Z$ we define also an antiunitary operator 
$j_k : \H_k\rightarrow \H_{-k}$ as follows:
\begin{equation}\label{defj0}
j_k = \left\{\begin{array}{lcl}
\Gamma J & \quad & \mathrm{if}\; j \equiv 1 \; (\mathrm{mod}\; 4) \\
J & \quad & \mathrm{if}\; j \equiv 3 \; (\mathrm{mod}\; 4)
\end{array}\right.
\end{equation}
(where on the r.h.s.~appropriate restrictions of $\Gamma$ and $J$  are understood). 
Now we can state the following
\begin{prop}\label{propprojU1odd}
$(\B,\H_0,D_0,\gamma_0,j_0)$ and,  for $k\neq 0$, 
$(\B,\H_k\oplus\H_{-k}, D_k\oplus D_{-k}, \gamma_k\oplus\gamma_{-k}, j_k\oplus j_{-k})$, 
are even real spectral triples of $KR$-dimension $j-1$. 
\end{prop}
\begin{proof}
 We check here only the commutation relations of the operators $D_k,\gamma_k,j_k$ 
and the first order condition. For rest of the proof see \cite{DS10}, proposition 4.4. 
In fact we check these relation only on the subspace $\H_0$, since the extension 
of the computations below to the general case $k\in\Z$ is straightforward.

For $j=3$ the result is already proved in \cite{DS10}. 
Let us consider now the case $j = 1$: $[\Gamma,J] = 0$, $j_0 = \Gamma J$. 
We need to check that: $j_0^2 = 1$, $j_0D_0 = D_0j_0$, 
$\gamma_0j_0 = j_0\gamma_0$. We have:
$$j_0^2 = \gamma_0 J\gamma_0 J = \Gamma J\Gamma J = \Gamma^2 J^2 = 1,$$
\begin{equation*}
\begin{split}
j_0D_0 & = {1\over 2}\Gamma J \Gamma[\Gamma, D] 
= {1\over 2}\left(\Gamma J D -\Gamma J \Gamma D\Gamma \right) \\
& = {1\over 2}\left(\Gamma JD-JD\Gamma  \right) = 
{1\over 2}\left(DJ\Gamma - \Gamma DJ\right) 
= {1\over 2}\left(D\Gamma J - \Gamma DJ \right) = D_0j_0,
\end{split}
\end{equation*}
$$\gamma_0j_0 = \Gamma\Gamma J = \Gamma J \Gamma = j_0\gamma_0.$$
Now the case $j = 5$: $[\Gamma,J] = 0$, $j_0 = \Gamma J$. 
We need to check that $j_0^2 = -1$,  $j_0D_0 = D_0j_0$, $\gamma_0j_0 = j_0\gamma_0$. 
Since the only difference with the previous case 
is that now $J^2 = -1$, the proof of the last two relations is the same as before. 
For the first one:
$$j_0^2 = \gamma_0 J\gamma_0 J = \Gamma J\Gamma J = \Gamma^2 J^2 = -1.$$
We are left with the proof of the proposition for $j = 7$. 
In this case we have $j_0 = J$, 
$J\Gamma = -\Gamma J$, and we have to check that $j_0^2 = 1$, $j_0D_0 = D_0j_0$, 
$\gamma_0j_0 = -j_0\gamma_0$. We have:
$$j_0^2 = J^2 = 1,$$
\begin{equation*}
\begin{split}
j_0D_0 & = {1\over 2}J\Gamma[\Gamma, D] = {1\over 2}\left(JD-J\Gamma D\Gamma \right)\\
& = {1\over 2}\left(DJ-\Gamma D\Gamma J \right) = D_0j_0,
\end{split}
\end{equation*}
$$\gamma_0j_0 = \Gamma J = -J\Gamma = -j_0\gamma_0.$$

The first order condition follows directly from \eqref{projcalc}, 
and from the fact that $\Gamma$ commutes with $D$ and with the elements of $\B$,
and it either commutes or anticommutes with $J$.
\end{proof}

\subsection{Projectable spectral triples: even case}\label{ssectprojU1even}
Now we extend the notion of projectable spectral triple to the even dimensional case. 
We give the following definition, which, as we shall see later, is consistent with 
the results obtained in the commutative (smooth) case \cite{AB98,A98}.
\begin{defin}
 An even $U(1)$-equivariant real spectral triple $(\A,\H,D,J,\gamma,\delta)$ 
is said to be  \emph{projectable} along the fibers if there exists a $\Z_2$ grading 
$\Gamma$ on $\H$, which satisfies the following conditions,
$$\Gamma^2 = 1,\quad\quad \Gamma^* = \Gamma,$$
$$[\Gamma,\pi(a)]=0 \quad\forall a\in\A,$$
$$[\Gamma,\delta] = 0,$$
$$\Gamma\gamma = -\gamma\Gamma,$$
$$\Gamma J = -J \Gamma.$$
If such a $\Gamma$ exists, we define the \emph{horizontal Dirac operator} $D_h$ by
$$D_h\equiv {1\over 2}\Gamma[\Gamma, D]$$
\end{defin}

Also in this case we can introduce the property of constant length fibers, 
similarly as in definition \ref{defU1constlength}. Consider now a projectable triple 
$(\A,\H,D,J,\gamma,\delta,\Gamma)$ of $KR$-dimension $j$, 
with $j$ even, and assume that it has fibers of constant length. 
Again $D_h$ preserves each $\H_k$, 
and the real structure intertwines $\H_k$ and $\H_{-k}$, $J\H_k\subseteq \H_{-k}$. 
In particular, it preserves $\H_0$. Let us denote, for any $k\in\Z$, by $D_k$, 
and $\gamma_k$ the restrictions to $\H_k$ of $D_h$ and $\Gamma$, respectively. 
Define an operator $\nu$ by $\nu = i\Gamma\gamma$.
 Then $\nu^* = \nu$ and $\nu^2 = 1$, and we can use it to split $\H_0$. 
In particular we obtain the following result.
\begin{prop}\label{propdim2st}
Decompose $\H_0$ as $\H_0 = \H_0^{(+)}\oplus\H_0^{(-)}$, where $\H_0^{(\pm)}$ are the 
$(\pm 1)$-eigenspaces of $\nu$. Then the horizontal Dirac operator $D_h$ 
preserves both the subspaces $\H_0^{(\pm)}$. 
Moreover, if we denote respectively by $D_0^{(\pm)}$ the restrictions of $D_h$ 
to $\H_0^{(\pm)}$, $(\B,\H_0^{(\pm)},D_0^{(\pm)})$ are spectral triples.
\end{prop}
\begin{proof}

Clearly $D_h$ preserves $\H_0$, so for the first part of the proposition 
that $D_0$ preserves $\H_0^{(\pm)}$ we need only to check that $[D_0,\nu] = 0$.
Indeed
\begin{equation*}
\begin{split}
[D_0,\nu] &= {1\over 2}[\Gamma[\Gamma, D],\Gamma\gamma] \\
&= {1\over 2}\left(D\Gamma\gamma + \Gamma\gamma\Gamma D\Gamma - \Gamma D\gamma
- \Gamma\gamma D \right) \\
& = {1\over 2}\left(\Gamma\gamma D + D\gamma\Gamma - D\gamma\Gamma - 
\Gamma\gamma D \right) = 0.
\end{split}
\end{equation*}
Next since $\Gamma$ and $\gamma$ commute with $\B$, and $[D,b]$ is bounded 
since $\B\subset \A$, we have that $[D_0,b]$ is bounded for any  $b\in\B$. 
Of course, both $\H_0^{(\pm)}$ are preserved by $\B$. So $(\B,\H^{(\pm)},D_0^{(\pm)})$, 
where $D_0^{(\pm)}$ are the restrictions of $D_0$, are spectral triples.

Moreover, it follows from remark 4 of \cite{DS10} that the Dirac operators $D_0^{(\pm)}$ 
have compact resolvent.
\end{proof}
\begin{rmk}\label{rmkorientchange}
 Let $(\A,\H,D,J,\gamma,\delta,\Gamma)$ as in the previous proposition. 
Notice that $\Gamma\nu = -\nu\Gamma$, so that $\Gamma\H_0^{(\pm)}\subset \H_0^{(\mp)}$. 
Since $\Gamma^2 = 1$, $\Gamma$ determines an isomorphism 
$\Gamma : \H_0^{(+)}\rightarrow \H_0^{(-)}$. Moreover, one can easily see that
 $D_h\Gamma = -\Gamma D_h$. So, $D_0^{(+)} = - D_0^{(-)}$ with respect to the isomorphism 
$\H_0^{(+)}\simeq \H_0^{(-)}$ determined by $\Gamma$. This is nothing else than the noncommutative 
counterpart of the fact that, in the smooth case, the two Dirac operators $D_0^{(\pm)}$ are associated 
to the same metric, but they differ by a different choice of orientation \cite{AB98,A98}. 
So we can say that the two triples $(\B,\H^{(\pm)},D_0^{(\pm)})$ differ only by 
 \emph{the choice of (the sign of) the orientation}.
$\quad\diamond$ \end{rmk}
Now we can check if the spectral triples on $\B$ given by the previous proposition are real. 
We start with the $KR$-dimension 2 case.
\begin{prop}\label{propdim2}
 Let $(\A,\H,D,J,\gamma,\delta,\Gamma)$ be a projectable real spectral triple of $KR$-dimension 2, 
fulfilling the condition of fibers of constant length. Then the antiunitary operator $\gamma J$ preserves 
both the subspaces $\H_0^{(\pm)}$. Moreover, if we denote by $j_0^{(\pm)}$ the restrictions 
of $\gamma J$ to $\H_0^{(\pm)}$ respectively, 
then $(\B,\H_0^{(\pm)},D_0^{(\pm)},j_0^{(\pm)})$ are real spectral triples 
of $KR$-dimension 1, and they differ just by the change of the sign of the orientation 
(see previous remark).
\end{prop}
\begin{proof}
We know that both $J$ and $\gamma$ preserve $\H_0$ and if $j_0$ is the restriction 
of $\gamma J$ to $\H_0$ we see that $j_0$ preserves $\H_0^{(\pm)}$, 
since $[j_0,\nu] = 0$. 
We have also that $j_0$ commutes with $\Gamma$, since the spectral triple is projectable and has $KR$-dimension 2. 
Moreover, compute 
$$D_0j_0 = {1\over 2}\left(D\gamma J- \Gamma D\Gamma\gamma J\right) = - j_0D_0,$$
as it should in $KR$-dimension 1.

Next since $\gamma$ commutes with $\A$, it is immediate that $j_0$ maps $\B$ 
to its commutant. Since $-J^2 = \gamma^2 = 1$ and $J\gamma = -\gamma J$, 
we have also 
$j_0^2 = 1$. 
Thus $j_0$, and also $j_0^{(\pm)}$, fulfill all the commutation relations required for a real structure of 
a real spectral triple of $KR$-dimension 1. Moreover, the first order condition follows from \eqref{projcalc}
and from the first order condition for the spectral triple over $\A$.

The last statement of the proposition follows from the fact that $\Gamma$ intertwines the two triples,
 as shown in remark \ref{rmkorientchange}.
\end{proof}

In order to extend the result of proposition \ref{propdim2} to higher dimensional even spectral triples
we give the following definition.
\begin{defin}\label{defJ0}
 Let $(\A,\H,D,J,\gamma,\delta,\Gamma)$ be an even dimensional projectable real spectral triple 
of KR-dimension $j$. Then we define a real structure $j_0$ on $\H_0$ by:
\begin{equation}\label{J0}
 \begin{array}{ccc|c|c|c}
  \mathrm{KR-dim} & \quad & 0 & 2 & 4 & 6 \\
  j_0  & \equiv & J & \gamma J & J & \gamma J
 \end{array}
\end{equation}
where the restriction of $\gamma$ and $J$ to $\H_0$ is understood.
\end{defin}
With this definition of $j_0$ we can prove the following result, which is a generalization 
of proposition \ref{propdim2}.
\begin{prop}\label{propdimgen}
 Let $(\A,\H,D,J,\gamma,\delta,\Gamma)$ be a projectable even real spectral triple 
of $KR$-dimension $j$, 
fulfilling the constant length fibers condition. Let $j_0 : \H_0\rightarrow \H_0$ be given by \eqref{J0}. 
Then $j_0$ restricts to $\H_0^{(\pm)}$, and $(\B,\H_0^{(\pm)},D_0^{(\pm)},j_0^{(\pm)})$ are real spectral
 triples of $KR$-dimension $(j-1)$. Moreover, they differ just by the change of the sign of the orientation.
\end{prop}
\begin{proof}
 We have already discussed the case $n = 2$. So we prove the proposition separately in the other 
three cases. All what we need to check is that $j_0^2 = \pm 1$ accordingly to $KR$-dimension $(j-1)$,
 that $[j_0,\nu]$ = 0, and that $D_0^{(\pm)}$ and $j_0^{(\pm)}$ satisfy the correct commutation relations
 (see table 1); the other properties (like the first order condition) are fulfilled for the same reasons 
as the previous proposition. The first condition is easily checked.

Let us check that $j_0$ commutes with $\nu$.
Let $j = 4$. Then $j_0 = J$ and
$$[j_0,\nu] = [J,i\Gamma\gamma] = [J,i\Gamma]\gamma = 0.$$
Let $j = 6$. Then $j_0 = \gamma J$ and
$$[j_0,\nu] = [\gamma J,i\Gamma\gamma] = \gamma[J,i\Gamma]\gamma = 0.$$
Finally, let $j = 0$. Then $j_0 = J$ and
$$[j_0,\nu] = [J,i\Gamma\gamma] = [J,i\Gamma]\gamma = 0.$$
Now we have to check the commutation relation between $j_0$ and $D_0$. But we notice that the
 commutation relations between $D_0$ and $j_0$ are the same of those between $j_0$ and $D$. 
So, if $j = 0$ or $j = 4$ then $j_0 = J$ and thus $j_0D_0 = D_0j_0$, and it is consistent 
with the requirements, respectively, of $KR$-dimension 7 and 3.if $j = 6$ then 
$j_0 D_0 = -D_0 j_0$, as it should be in $KR$-dimension 5.
\end{proof}
We conclude this section pointing out that, as for the odd dimensional case, 
for $k\neq 0$ we can define real spectral triples 
$(\B,\H_k\oplus \H_{-k}, D_k\oplus D_{-k}, j_k\oplus j_{-k})$
 of $KR$-dimension $j-1$,
simply extending the construction discussed above for the $k=0$ case.

\section{Twisted spectral triples}\label{tst}

In this section we shall use the strong connection on the circle bundle to twist the horizontal Dirac 
operator. We discuss here only the case of even spectral triples, for the reason that
the odd dimensional case, which has less properties to check, follows immediately. The latter 
was discussed also in details in \cite{DS10}.

\subsection{Twisted Dirac operators}\label{tDo}
Let the data $(\A,\H,J,D,\gamma,\delta,\Gamma)$ be a $U(1)$-equivariant projectable even real spectral
 triple, with constant length fibers, over a quantum principal $U(1)$-bundle $\B \hookrightarrow \A$, and assume the Dirac calculus $\Omega^1_D(\A)$ to be compatible with the de Rham calculus on $H$ (here $H$ is the coordinate algebra of $U(1)$).
Let $\omega\in\Omega^1_D(\A)$ be a strong connection form.
Let us notice that, for any $k\in\Z$, the set $\A^{(k)}$, acting on the right on $\H$ via the right action induced by the real structure $J$ 
(see eq. \eqref{JaJ}), can be regarded as a set of bounded left
 $\B$-linear maps between $\H_0^{(\pm)}$ and $\H_k^{(\pm)}$ (where, we recall, 
the $(\pm)$-decomposition is done accordingly to $\nu^2 = 1$). 
We introduce a new notation: we write the action of $\A^{(k)}$ on the right, that is $m(h) \equiv hm$. 
Then the $\B$-linearity reads $$(bh)a = b(ha).$$
%
We observe that from the fact that 
$\A$ is a principal comodule algebra, it follows that
for any $k\in\Z$, $\A^{(k)}$ is a projective $\B$-module,
which can be easily seen to be finitely generated.\\
This structure is compatible with the right $\B$-module structure induced on $\H$ 
by the real structure $J$:
$$(ba)(h) = a(hb),\quad\;\;\forall a\in \A^{(k)}, b\in\B, h\in\H,$$
which in the right-handed notation becomes
$$h(ba) = (hb)a.$$

Consider the left $H$-comodule $(V,\rho_L)$, where $V = \C$ and $\rho_L(\lambda) = z\otimes \lambda$. For any $k\in\Z$ define the $H$-comodule $(V^k,\rho_L^k)$ by setting $V^k = \C$, $\rho_L^k(\lambda) = z^k\otimes\lambda.$ Then it is straightforward to see that $\A^{(k)}$ is isomorphic to $\A\Box_H V^k$ (where $\Box_H$ denotes the cotensor product over $H$).
It follows that $\A^{(k)}$ is a quantum associated bundle (see e.g. \cite{BF12}).

We introduce now two assumptions:
\begin{description}
\item[(i)] 
$\H_0 \A^{(k)} \equiv \A^{(k)}(\H_0)$ is dense in $\H_k$, where $\A^{(k)}(\H_0)$ is the linear span of elements $a(h)$, $a\in \A^{(k)}$, $h\in\H_0$;
\item[(ii)] the multiplication map from $\H_0\otimes_\B \A^{(k)}$ to $\H_k$ is an isomorphism.
\end{description}

The first order condition \eqref{firstordcond} implies that there is a right action of $\Omega^1_D(\B)$ on $H$, given by:
\begin{equation}\label{homega}
h\omega = - J\omega^*J^{-1}h\quad\;\;\forall \omega\in\Omega^1_D(\B),
\end{equation}
where $\omega^*$ is the adjoint of $\omega$, s.t. $([D,b])^* = -[D,b^*]$ and
$$h[D,b] = D(hb) - (Dh)b.$$
Such an action is clearly left $\B$-linear. Moreover, it induces a left action of 
$\Omega^1_D(\B)$ on $\A^{(k)}$, and $\Omega^1_D(\B)\A^{(k)}$ is just the space of all compositions
 $a\circ \omega$ of left $\B$-linear maps. For further details we refer to \cite{DS10}.
\begin{rmk}\label{rmkj0}
 As we have seen in the previous sections, the real structure which makes the triple over $\B$ 
a real spectral triple is not always the simple restriction of $J$ to $\H_0$. Nevertheless both $J$ 
and the collection of $j_k$ induce the same right action of $\A$ on $\H$. Instead 
in order to define the right action of $\Omega^1_D(\B)$ on $\H_0$ 
we shall use the real structure $j_0$, to be consistent with \ref{homega}.
In this case it would not be correct to use $J$, since in general its commutation relation with $D_0$ 
is different from that of $j_0$, and the representation of a differential form involves the Dirac operator.
$\quad\diamond$ \end{rmk}

Now we come to the construction of twisted spectral triples over $\B$.
Using the $D_0^{(\pm)}$-connection $\nabla_\omega$ on the left $\B$-module $\A^{(k)}$, we twist the Dirac operators $D_0^{(\pm)}$ as follows
\begin{equation}\label{eqDomegak}
D_\omega^{(k,\pm)}(ha) = (D_0^{(\pm)}h)a + h \nabla_\omega (a)
\end{equation}
Note that as explained in \cite{DS10}, due to (i) and (ii), $D_\omega^{(k,\pm)}$ depends only 
on the product $ha\in \H_k$.

Let now $D_\omega^{(\pm)}$ be the respective closures of the direct sums of the two families
$D_\omega^{(k,\pm)}$. In order to regard them as two twisted Dirac operators 
on $\H$ we need another assumption:\\ 
{\bf (iii)} given $Z$, there exists bounded selfadjoint $Z'$
such that
$$
(Zh)a=Z'(ha),\quad \forall h\in\H , a\in \A.
$$  
Note that it is weaker than assuming that $[Z,J\A J^{-1}]=0$.

\begin{prop}
The  operators $D_\omega^{(\pm)}$ are selfadjoint if $\omega$ is a selfadjoint one form, 
and they have bounded commutators with all the elements of $\A$.
\end{prop}
\begin{proof}
Take $h\in\H_0^{(\pm)}$ and $p\in\A^{(k)}$ such that $hp$ is in the domain of $D_\omega^{(\pm)}$. 
Then we have:
\begin{equation}\label{eqDomega}
\begin{split}
D_\omega^{(\pm)} & = (D_0^{(\pm)}h)p + h[D,p] - khp\omega \\
& = (D_0^{(\pm)}h)p + [D,j_0p^*j_0^{-1}]h + j_0\omega^*j_0^{-1}khp \\
& = D(hp) + ((D_0^{(\pm)} - D)h)p + j_0\omega^*j_0^{-1}\delta(hp) \\
& = (D + j_0\omega^*j_0^{-1}\delta - Z')(hp).
\end{split}
\end{equation}
From \eqref{eqDomega} follows, by standard results of functional analysis, the selfadjointness of $D_\omega$. Next, $D$ has bounded commutator with each $a\in\A$; $\omega$ is a one-form and so, due to the first order condition, the commutator of the second term of \eqref{eqDomega} with $a\in\A$ is simply $j_0\omega^*j_0^{-1}\delta(a)$ and hence is bounded. The commutator with the first term is bounded simply because it is the commutator of two bounded operators. We conclude then that $[D_\omega,a]$ is bounded $\forall a\in\A$.
\end{proof}
This shows that $(\B,\H^{(\pm)},D_\omega^{(\pm)})$ is a (twisted) spectral triple (non necessarily real).
Notice that the $\pm$ label corresponds just to different choice of the orientation, as follows from 
remark \ref{rmkorientchange}.
Moreover, since $D_0^{(\pm)}$-connection $\nabla_\omega$ on the left $\B$-module $\A$, 
restricts to $D_0^{(\pm)}$-connections $\nabla_\omega$ on each $\B$-module $\A^{(k)}$,
it is clear that $(\B,\H_k^{(\pm)},D_\omega^{(k,\pm)})$, $k\in\Z$, 
is a family of twisted spectral triples, obtained by suitable restrictions.

\begin{cor}
The ``full'' Dirac operator $D_\omega:=D_\omega^{(+)}\oplus D_\omega^{(-)}$  is selfadjoint if $\omega$ is selfadjoint and has bounded commutator with all the elements of the algebra $\A$.
\end{cor}

Finally, with  $Z$ as in definition \ref{defU1constlength}
and the assumption (iii) above, we define
$$\mathcal{D}_\omega := \Gamma\delta + D_\omega.$$
\begin{prop}
$(\A,\H,\mathcal{D}_\omega)$ is a projectable spectral triple with constant length fibers and the horizontal part of the operator $\mathcal{D}_\omega$ coincides with $D_\omega$.
\end{prop}
\begin{proof}
See proof of proposition 5.8 in \cite{DS10}.
\end{proof}

Moreover, as in \cite{DS10} we introduce the following notion of \emph{compatibility}.
\begin{defin}\label{defcomp}
We say that a strong connection $\omega$ is \emph{compatible} with a Dirac operator $D$ if $D_\omega$ and $D_h$ coincide on a dense subset of $\H$.
\end{defin}
Thus given a projectable spectral triple with constant length
fibers and a compatible strong connection, 
the operators $D^{(k)}$ are just $D_0$ twisted by $\nabla_\omega$ on $\A^{(k)}$.

\section{Noncommutative tori}

We will show how the canonical flat spectral triples over $n$-dimensional noncommutative tori are
 projectable and we will work out explicit formulae for the twisted Dirac operators
for the noncommutative 2-torus as a bundle over the circle $S^1$ and the noncommutative 4-torus as a bundle over a noncommutative 3-torus. (The case of 3-torus over the 2-torus is discussed in \cite{DS10}). 

\subsection{Quantum principal $U(1)$-bundle $C^\infty(S^1)\hookrightarrow \T^2_\theta$}

Let $\A = \T^2_\theta$ be the unital smooth algebra of the noncommutative 2-torus, 
with two unitary generators $U,V$ and the commutation relation $UV = e^{2\pi i\theta}VU$ 
($\theta$ irrational).
Let $\delta_1, \delta_2$ be two derivations of $\A$ determined by
$$\delta_2(U) = 0,\quad \delta_2(V) = V; \quad\quad \delta_1(U) = U,\quad \delta_1(V) = 0.$$
Consider the $U(1)$ action on $\A$ given, at infinitesimal level, by $\delta_2$.
Then the invariant subalgebra $\B$ is the (commutative) algebra generated by $U$ and we can identify $\B$ with $C^\infty(S^1)$. 
Now let $\H = \H_\tau\otimes \C^2$, where $\H_\tau $ is the GNS Hilbert space associated to the unique tracial state on $\A$.
The derivations $\delta_j$ can be implemented (via the commutator) by selfadjoint operators on $\H$ 
with integer spectrum, which for simplicity we denote by the same symbol $\delta_j$.
Let $\H_k$, $k\in \Z$, be the $k$-eigenspace of $\delta_2$ in $\H$. 

Consider the standard ``flat'' Dirac operator on $\H$,
$$D = \sum_{i=1}^2 \sigma^i\delta_i,$$
where the $\sigma^i$ are the Pauli matrices. There is also the orientation $\Z_2$-grading
$$\gamma = \id\otimes \sigma_3$$
and the real structure
$$J = J_0\otimes (i\sigma^2\circ c.c.),$$
where $J_0 : \H_\tau\rightarrow \H_\tau$ is the Tomita-Takesaki antiunitary involution 
and $c.c.$ denotes the complex conjugation. 
\begin{prop}
There exists a unique (up to a sign) operator 
$$\Gamma = \pm\,\id\otimes \sigma^2 :\H\rightarrow \H,$$ 
such that $(\A,\H,D,J,\gamma,\delta_2,\Gamma)$ is a projectable real spectral triple.
\end{prop}
\begin{proof}
Since $\Gamma$ has to commute with $\pi(\A)$, and $\H_\tau$ is an irreducible
 representation of $\A$,
 we have that the most general form of an admissible $\Gamma$ is:
$\Gamma = \alpha_0\cdot\id + \sum_{i=1}^3 \alpha_i\sigma^i$
with $\alpha_i\in\C$. And using $\Gamma\gamma = -\gamma\Gamma$ 
we see immediately that $\alpha_0 = 0$ and $\alpha_3 = 0$. 
Next, from $\Gamma^2 = -1$ we obtain: 
$$\alpha_1^2 + \alpha_2^2 = 1.$$
The last condition to impose is \ref{za}. In fact, already \ref{projcalc} suffices to get, 
for any $v\in\H$,
$$
\sigma^1 U \delta_1(v) = (\alpha_1\alpha_2\sigma^2 -\alpha_2^2\sigma^1)U\delta_1(v).$$
This implies that the only solution is $\alpha_2 = \pm 1$. 
Moreover, $\Gamma = \pm \sigma^2$ 
is consistent with the commutation relation $J\Gamma = -\Gamma J$.
\end{proof}
It follows that $D_h = \sigma^1\delta_1$, and that we can identify both $\H_0^{(\pm)}$ 
with $L^2(S^1,d\varphi)$; then the restriction of $D_h$ to $\H_0^{(\pm)}$ is given by:
$$D_0^{\pm} = \pm i {d\over d\varphi}.$$
Also, the real structure $j_0^{(\pm)}$ on $\H_0^{(\pm)}$ is simply the complex conjugation.

\subsubsection*{Twisted Dirac operators}
If we set $D_v = \sigma^2\delta_2$ we see that the spectral triple discussed above 
has the constant length fibers property and in addition the operator $Z$ turns out to be zero. 
In order to build the twisted Dirac operators we first need a strong connection 
over $\A$. We have:
\begin{lemma}
 A $U(1)$ selfadjoint strong connection over $\A$ is a one-form
$$\omega = \sigma^2 + \sigma^1\omega_1,$$
where $\omega_1$ is a selfadjoint element of $\B$.
\end{lemma}
\begin{proof}
Let $\omega\in\Omega^1_D(A)$; then $\omega$ can be written as:
$$\omega = \sum_i a_i[D,c_i]$$
with $a_i,c_i\in\A$. This implies that we can generally write $\omega$ as
$$\omega = \sum_i \sigma^i\omega_i$$
with $\omega_i\in\A$. In order to be a strong connection, $\omega$ has to fulfill the 
properties of definition \ref{defcon}. In particular we need $[\delta,\omega] = 0$, 
and this implies $\omega_i\in\B$.
 Also, $\omega_3$ should be equal to zero, since we cannot obtain an operator such as 
$\sigma^3\omega^3$ from a commutator $[D,a]$. Thus we are left with a connection 
of the form:
$$\omega = \sigma^1\omega_1 + \sigma^2\omega_2.$$
Now notice that, for $j = 1,2$, we can write the Pauli matrices $\sigma^j$ as: 
$\sigma^j = U_j^{-1}[D,U_j]$ (where $U_1 = U$, $U_2 = V$). Then $\omega$ becomes:
$$\omega = \sum_{j=1}^2 \omega_j U_j^{-1}[D,U_j].$$
But now using the second condition of definition \ref{defcon} we obtain: $\omega_2 = 1$. 
Thus the most general $U(1)$ strong connection on $\A$ is:
$$\omega = \sigma^2 + \sigma^1\omega_1, \quad\quad \omega_1\in\B.$$
\end{proof}

Now we can compute the Dirac operator $D_\omega$ twisted by the strong connection 
$\omega$ following the construction described in the previous section.
\begin{prop}\label{propDomegaT2}
For any selfadjoint $U(1)$ strong connection $\omega$, the associated Dirac operator 
$D_\omega$ has the form
$$D_\omega = D_h - \sigma^1j_0\omega_1j_0^{-1}\delta_2.$$
\end{prop}
\begin{proof}
By the results in sect. \ref{ssectprojU1even} the projectability of the spectral triple $(\A,\H,D,J,\delta)$
implies that there are two spectral triples over the invariant subalgebra $\B$, 
which in this case is simply the algebra of smooth functions on $S^1$. 
They are given by $(\B,\H_0^{(\pm)},D_0^{(\pm)},j_0)$. 
In order to fix the conventions, we say that on $\H_0^{(+)}$ the Dirac operator $D_0$ is given 
by $-\delta_1$, while on $\H_0^{(-)}$ it is given by $\delta_1$ (note that $\nu = -\sigma^1$, 
and thus $\sigma^1$ is diagonal with respect to the decomposition $\H_0 = \H_0^{(+)}\oplus\H_0^{(-)}$). 
Now we have to compute the twisted operators $D_{\omega,k}^{(\pm)}$ on $\H_k^{(\pm)}$. 
Take $h_0\in\H_0^{(+)}$. Then, for any $a\in\A^{(k)}$, we have:
$$D_{\omega,k}^{(+)}(h_0a) = (D_0^{(+)}h_0)a - h_0\nabla_\omega(a)
= -\delta_1(h_0a)+ kh_0a\omega_1.$$
Thus, if we take $h\in\H_k^{(+)}$ we see that the action of the twisted Dirac operator is given by
$$D_{\omega,k}^{(+)}(h) = -\delta_1(h) + j_0\omega_1j_0^{-1}\delta_2(h).$$
In the same way, one obtains that, for $h\in\H_k^{(-)}$,
$$D_{\omega,k}^{(-)}(h) = \delta_1(h) - j_0\omega_1j_0^{-1}\delta_2(h).$$
If now we put them together, and we consider the collection of all of them for any $k\in\Z$, we get that the full twisted Dirac operator $D_\omega$ is given, as an operator on $\H$, by
$$D_\omega = \sigma^1\delta_1 - \sigma^1j_0\omega_1j_0^{-1}\delta_2,$$
which is equal to $D_h - \sigma^1j_0\omega_1j_0^{-1}\delta_2$.
\end{proof}
\begin{cor}
 The only connection compatible with $D$, i.e. with the fully $U(1)^2$-equivariant Dirac operator, on the noncommutative 2-torus is $\omega = \sigma^2$.
\end{cor}
\begin{proof}
 It follows from previous lemma and definition \ref{defcomp}.
\end{proof}
Now we can compute, given any strong connection $\omega$, the general form 
of a Dirac operator $D^{(\omega)}$ compatible with such a connection.
\begin{prop}
Let $\omega = \sigma^2 + \sigma^1\omega_1$ be a selfadjoint connection. 
Then the following Dirac operator,
$$D^{(\omega)} = D - \sigma^1j_0\omega_1j_0^{-1}\delta_2,$$
is compatible with $\omega$.
\end{prop}
\begin{proof}
It follows from definition \ref{defcomp} together with the computation of proposition 
\ref{propDomegaT2}.
\end{proof}

\subsection{The principal extension $\T^3_{\theta'}\hookrightarrow\T^4_\theta$}
Let  $\A$ be the unital (smooth) algebra of the noncommutative $4$-torus, generated
 by four unitaries $U_1, U_2,U_3,U_4$ with the commutation relations 
$U_iU_j = e^{2\pi i\theta_{ij}}U_jU_i$, where $\theta_{ij}$ is an antisymmetric matrix
 with no rational entries and no rational relation between them. On $\A$ 
there is the canonical action of $U(1)^4$, whose generators are the derivations 
$\delta_j$,
$$\delta_i(U_j) = \delta_{ij}U_j.$$
As $U(1)$ quantum principal bundle structure we take the one given by the choice 
$\delta = \delta_4$, and we assume the relative spin structure to be the trivial one. 
Thus the invariant subalgebra $\B$ is the algebra generated by $U_1,\ldots,U_3$ 
and is isomorphic to the algebra of a noncommutative 3-torus.

We recall briefly the structure of (one of the) flat $\T^4$-equivariant spectral triples 
over $\T^4_\theta$. The commutation relations in $KR$-dimension 4 are the following 
ones:
\begin{equation}\label{KR4}
 J^2 = -1,\quad\; JD = DJ,\quad\; J\gamma = \gamma J.
\end{equation}
In order to work out explicitly the operators, it is useful to recall the structure 
of the Clifford algebra $\Cl(4)$ (so that we can fix the notation).

\subsubsection*{The Clifford algebra $\Cl(4)$}
The Clifford algebra $\Cl(4)$ is generated by four elements, 
$\gamma^1,\ldots,\gamma^4$, with the relations
\begin{equation}\label{Cl4}
 \begin{array}{l}
{\gamma^i}^2 = 1,\\
\gamma^i\gamma^j = -\gamma^j\gamma^i \;\;\;\;\mathrm{for}\; i\neq j, \\
{\gamma^i}^* = \gamma^i.
\end{array}
\end{equation}
We can represent the $\gamma^i$'s as $4\times 4$ matrices, related to the Dirac
matrices. In the so-called Dirac representation we can write the matrices $\gamma^i$ 
as:
\begin{equation}\label{gammam}
 \gamma^4 = \left(\begin{array}{cc}
1 & 0 \\ 0 & -1
\end{array}\right),\quad\quad
 \gamma^j = \left(\begin{array}{cc}
0 & i\sigma^j \\ -i\sigma^j & 0
\end{array}\right).
\end{equation}
Moreover, we can define a matrix $\gamma^5\equiv 
\gamma^1\gamma^2\gamma^3\gamma^4$ which satisfies $\gamma^5\gamma^j = 
-\gamma^j\gamma^5$, ${\gamma^5}^2 = 1$ and ${\gamma^5}^* = \gamma^5$; 
in Dirac representation:
\begin{equation*}
 \gamma^5 = \left(\begin{array}{cc}
0 & 1 \\ 1 & 0
\end{array}\right).
\end{equation*}

We recall, also, that using the Dirac matrices we can write down a basis for 
$M_4(\C)$. In particular, if we define $\sigma^{ij} \equiv [\gamma^i,\gamma^j]$ 
(for $i,j = 1,\ldots,4$), then the basis is given by:
\begin{equation}\label{M4Cbasis}
 \begin{array}{lcl}
\id,\;\;\gamma^5, & & \\
\gamma^i & \quad & i=1,\ldots,4, \\
\gamma^5\gamma^i & \quad &  i=1,\ldots,4, \\
\sigma^{ij} & \quad & i < j.
 \end{array}
\end{equation}

\subsubsection*{A projectable spectral triple}
Now, let $\H_\tau$ be the GNS Hilbert space associated to the canonical trace 
$\tau$ on $\A$ \cite{Bondia}. Define $\H = \H_\tau\otimes \C^4$. 
We consider the usual flat Dirac operator \cite{Bondia}:
$$D = \sum_{j=1}^4 \gamma^j\delta_j.$$
Then we can take the orientation $\Z_2$-grading to be $\gamma = \gamma^5$. 
To define $J$, we recall that it is related to the charge conjugation operator; 
so we take $$J = J_0\otimes(\gamma^4\gamma^2\circ c.c.),$$
where $J_0 : \H_\tau\rightarrow \H_\tau$ is the Tomita-Takesaki involution 
and $c.c.$ denotes the complex conjugation. Then one can see that the spectral triple 
$(\A,\H,D,J,\gamma)$ satisfies the relations \eqref{KR4}, 
and it is also a $U(1)$-equivariant spectral triple, which is projectable:
\begin{prop}
 The unique operators $\Gamma : \H\rightarrow\H$, such that 
$(\A,\H,D,J,\gamma,\delta_4,\Gamma)$ is a projectable real spectral triple 
with projectable differential calculus, are $\Gamma = \pm\id\otimes \gamma^4$.
\end{prop}
\begin{proof}
 Since $[\Gamma,\pi(a)] = [\Gamma,\delta] = 0$ for all $a\in\A$, $\Gamma$ must be 
of the form $\Gamma = \id\otimes A$ for some matrix $A\in M_4(\C)$. Then using the 
fact that \eqref{M4Cbasis} give a basis of $M_4(\C)$, we can write $\Gamma$ as
$$\Gamma = a + b\gamma^5 + \sum_j c_j\gamma^j + \sum_j d_j\gamma^5\gamma^j 
+ \sum_{i,j}e_{ij}\sigma^{ij}.$$
From $\Gamma\gamma = -\gamma\Gamma$ we deduce $a = b = e_{ij} = 0$. 
Thus we are left with
$$\Gamma = \sum_j(\alpha_j\gamma^j + \beta_j\gamma^5\gamma^j),
\quad\;\alpha_j,\beta_j\in\C,$$
where $\alpha_j\in\R$ and $\beta_j\in i\R$, as follows from the condition 
$\Gamma = \Gamma^*$. This implies that we can write $\Gamma^2$ as:
$$\Gamma^2 = \sum_j \alpha_j^2 + \sum_{i\neq j}
2\alpha_i\beta_j\gamma^5\gamma^i\gamma^j - \sum_j\beta_j^2.$$
Next, using the condition $\Gamma^2 = -1$, we deduce:
\begin{equation}\label{Gammacond1}
\left\{\begin{array}{l}
 \displaystyle \alpha_i\beta_j = \alpha_j\beta_i \quad\quad\forall i\neq j \\ \\
 \displaystyle \sum_j (\alpha_j^2 - \beta_j^2) = 1
\end{array}\right.
\end{equation}

We have now to impose that $Z$ commutes with $\A$; in particular we shall require 
that $[D,b] = [D_h,b]$ for all $b\in\B$. Let us compute, first of all, 
$D_h = {1\over 2}\Gamma[D,\Gamma]$ (we use the Einstein convention 
for the sum over repeated indices):
\begin{equation}\label{Dh4t}
 \begin{split}
  D_h & = {1\over 2}(\alpha_i\gamma^i + \beta_i\gamma^5\gamma^i)[\gamma^j\delta_j, \alpha_k\gamma^k + \beta_k\gamma^5\gamma^k] \\
& = {1\over 2}(\alpha_i\gamma^i + \beta_i\gamma^5\gamma^i)(\alpha_k\sigma^{jk}\delta_j - 2\beta_k\gamma^5\delta_k) \\
& = \alpha_j\alpha_k\gamma^k\delta_j - \alpha_j\alpha_j\gamma^k\delta_k + \eps_{ijkl}\alpha_i\alpha_k\gamma^5\gamma^l\delta_j \\
& -\alpha_i\beta_k\gamma^i\gamma^5\delta_k + \eps_{ijkl}\beta_i\alpha_k\gamma^l\delta_j + \beta_i\beta_k\gamma^i\delta_k.
 \end{split}
\end{equation}
And now, from the condition $[D,b] = [D_h,b]$, using $[\delta_4,\B] = 0$ 
and the linear independence of the sixteen generators \eqref{M4Cbasis}, we get:
\begin{equation}\label{Gammacond2}
\left\{\begin{array}{l}
\displaystyle \eps_{ijkl}\alpha_i\alpha_k\gamma^5\gamma^l\delta_j - \alpha_i\alpha_k\gamma^i\gamma^5\delta_k = 0 \\ \\
\displaystyle \sum_{j\neq k} (\alpha_j\alpha_k\gamma^k\delta_j + \beta_i\beta_k\gamma^j\delta_k) + \eps_{ijkl}\beta_i\alpha_k\gamma^l\delta_j =  0 \\
\displaystyle \sum_{j\neq k}-\alpha_j^2\gamma^k\delta_k + \sum_i\beta_i^2\gamma^i\delta_i = \sum_{j=1}^3\gamma^i\delta_i.
\end{array}\right.
\end{equation}
The last condition implies:
\begin{equation}\label{Gammacond3}
 \left\{\begin{array}{l}
\displaystyle \beta_4^2 = \sum_{j = 1}^3 \alpha_j^2 \\
\displaystyle \mathrm{for}\;i\neq 4,\;\; \sum_{j\neq i} - \alpha_j^2 + \beta_i^2 = 1.
\end{array}\right.
\end{equation}
If now we use \eqref{Gammacond3} to compute $\sum_j\beta_j^2$ we get:
\begin{equation}\label{Gammacond4}
\sum_j\beta_j^2 = \sum_{j=1}^3\alpha_j^2 + \sum_{i=1}^3\left(1 + \sum_{j\neq i}\alpha_j^2 \right).
\end{equation}
Comparing \eqref{Gammacond4} with the second equation of \eqref{Gammacond1} 
we obtain the following relation:
\begin{equation}\label{Gammacond5}
\alpha_4^2 + {1\over 2}\beta_4^2 = 1.
\end{equation}
Now, we know that $\alpha_j\in\R$ and $\beta_j\in i\R$ (therefore $\beta_j^2 \leq 0$). 
Thus, from \eqref{Gammacond5} we obtain $\alpha_4^2 \geq 1$, while from the second 
equation of \eqref{Gammacond1} we get $\alpha_4^2 \leq 1$. So the only solutions are 
$\alpha_4 = \pm 1$, $\alpha_j = \beta_j = \beta_4 = 0$ for $j=1,2,3$. It is easy to see 
that such solutions fulfill all the other conditions of \eqref{Gammacond1}, 
\eqref{Gammacond2}. We conclude that the unique solutions for $\Gamma$ are 
$\Gamma = \pm \gamma^4$.
\end{proof}
Now we take one of the two solutions of the previous proposition, say 
$\Gamma = \gamma^4$. Then:
$$D_h = \sum_{i=1}^3\gamma^i\delta_i,\quad\quad D_v = \gamma^4\delta_4.$$
In particular the spectral triple has the ``constant length fibers'' property and, 
$Z = 0$. Now we can build the ``3-dimensional orientation'':
$$\nu = i\Gamma\gamma = i\gamma^5\gamma^4 = i\gamma^1\gamma^2\gamma^3.$$
We have $\nu^2 = 1$, $\nu^* = \nu$ as it should be. In $2\times 2$ matrix notation 
$\nu$ is given by:
\begin{equation*}
 \nu = \left(\begin{array}{cc} 0 & -i \\ i & 0 \end{array}\right).
\end{equation*}
It is useful to write the down action of $D_h$ on the two eigenspaces of $\nu$. 
Clearly it is enough to know the action of the matrices $\gamma_j$ ($j = 1,2,3$). 
Let us consider the $0$-eigenspace $\H_0$ of $\delta_4$, and decompose it 
accordingly to $\gamma$: $\H_0 = \H_0^+\oplus\H_0^-$. Then any vector can be 
written as $v = v_+\oplus v_-$. Moreover, since $\Gamma = \gamma_4$ is 
an intertwiner between $\H_0^\pm$, these two spaces are isomorphic. 
If $\H_0 = \H_0^{(+)}\oplus \H_0^{(-)}$ accordingly to $\nu$, then:
\begin{equation}
 \begin{array}{l}
  v\in\H_0^{(+)} \;\Rightarrow v\;\mathrm{is\;of\;the\;form}\;\; v = w\oplus(-iw) \\
  v\in\H_0^{(-)} \;\Rightarrow v\;\mathrm{is\;of\;the\;form}\;\; v = w\oplus iw
 \end{array}
\end{equation}
for some $w\in\H^+$. Using \eqref{gammam}, we see that, for $j = 1,2,3$, 
$\gamma^j$ acts as $\pm\sigma^j$ on $\H_0^{(\mp)}$. We summarize these results
 in the following lemma.
\begin{lemma}\label{lemmaNC4T}
Each Hilbert space $\H_0^{(\pm)}$ is isomorphic to $\H_\tau\otimes\C^2$,
 where $\H_\tau$ is the GNS Hilbert space associated to the canonical tracial state 
on $\B = \T^3_\theta$. Furthermore, the $\gamma$ matrices, as operators on $\H$, 
when restricted to $\H_0^{(\pm)}$ act as $\mp \sigma^j$.
\end{lemma}
Thus both the spectral triples are isomorphic to the ``standard'' one \cite{DS10} 
on the noncommutative 3-torus, with Dirac operators
$$D_0^{(\pm)} = \mp\sum_{j = 1}^3\sigma^j\delta_j.$$

Now we have to discuss the real structure. Since $J$ is antiunitary, we see that 
$[J,i\Gamma] = 0$. And, since $J\gamma = \gamma J$ and $J^2 = -1$ we can take 
$j_0^{(\pm)} = J$ (restricted to $\H_0^{(\pm)}$) and obtain that 
$(\B,\H_0^{(\pm)},D_0^{(\pm)},j_0^{(\pm)})$ are real spectral triples of KR-dimension 3.

\subsubsection*{Twisted Dirac operators}
As we have done for the noncommutative 2-torus, we can construct now the twisted 
Dirac operators. We present here only the main results; the proofs are omitted, 
since they are a direct generalization of the proofs of the analogue results in the 
2-dimensional case.
\begin{lemma}
 A $U(1)$ selfadjoint strong connection over $\A$ is a one-form
$$\omega = \gamma^4 + \sum_{j=1}^3 \gamma^j\omega_j,$$
where $\omega_j$ are selfadjoint elements of $\B$.
\end{lemma}
Given a strong connection $\omega$, the twisted Dirac operator is given by the 
following
\begin{prop}
For any selfadjoint $U(1)$ strong connection $\omega$ the associated Dirac operator 
$D_\omega$ has the form
$$D_\omega = D_h - \sum_{j=1}^3 \gamma^j J\omega_jJ^{-1}\delta_4.$$
\end{prop}
\begin{cor}
The only connection compatible with $D$ is $\omega = \gamma^4$. Moreover, 
let $\displaystyle\omega = \gamma^4 + \sum_{j=1}^3 \gamma^j\omega_j$ be 
a selfadjoint strong connection. Then the following Dirac operator,
$$\mathcal{D}_\omega = D - \sum_{j=1}^3\gamma^jJ\omega_jJ^{-1}\delta_4,$$
is compatible with $\omega$.
\end{cor}


\section{Theta deformations. $S^3_\theta$ as a $U(1)$-bundle over $S^2$}

Let $\lambda$ be a complex number of module one, which is not a
root of unity. 
A general construction of the Dirac operator and isospectral spectral triple on 
a Rieffel deformation quantization $M_\theta$ \cite{R} of any compact spin Riemannian
manifold M whose isometry group has rank $\geq 2$ has been described in \cite{cl01}.
Such theta-deformed spectral triples exist for instance over 
spheres $S^{N-1}$ and over $\bR^N$, with $N\geq 4$.
%
%
We will be interested in the lowest dimensional compact case of $S^3_\theta$, 
known also as a Matsumoto sphere, which will play the role of the total space of our
U(1) bundle, and will give explicit formula for the theta deformed Dirac operator. 
Before however we briefly recall  an equivalent but more `functorial'  realization of the 
theta-deformed 
spectral triples as introduced in  \cite{CD-V}, which we shall need in order to 
assure that our various assumptions made in sect. \ref{tst} are satisfied.

We restrict to the case of $S^{3}_\theta$. We denote its generators by $z_{m,\te}$, 
while the usual generators (complex coordinates) of $S^{3}$ by ${z_m}$
($z_{m,}=z_{m,0}$).
Let $\alpha$ be the  standard action of $\Tb^2$ on 
$\Tb^2$ on $S^{3}$ which consist of multiplying the generators
$z_{m}$, $m=1, 2$ by the relevant generators $u_m$ of $\Tb^2$. 
We denote by the same symbol $\alpha$ also the pullback action on $C^\infty(S^{3})$.
Let $\Tb^2_{\theta}$  be the well known noncommutative torus 
and let $\beta $  denote the usual action of $\Tb^2$ on 
$C^\infty(\Tb^2_{\theta})$, which consist of multiplying the generators
$u_{m,\te}$, $m=1, 2$; by the relevant $u_m$.
Then there is a `splitting'  isomorphism 
\begin{equation}\label{splitiso}
\kappa:C^\infty(S^{3}_{\theta}) \approx 
\left(
C^\infty(S^{3})\widehat\otimes C^\infty(\Tb^2_{\theta})
\right)^{ \alpha\otimes \beta^{-1}},
\end{equation}
given on the generators by 
\begin{equation}\label{split}
\kappa (z_{m,\te})={z_m}\otimes	u_{m,\te}, \quad m=1, 2;
\end{equation}
with the fixed point subalgebra of the 
action $\alpha\otimes\beta^{-1}$ of $\Tb^2$.
%
Here $\widehat\otimes$ denotes a suitable completion of $\otimes$.

Next the above construction was extended in \cite{CD-V} from theta-deformed manifolds  
to those vector bundles over $M$ which admit a direct lift of the action $\alpha$ of   $\Tb^2$ 
\begin{equation}\label{splitisovec}
\kappa:C^\infty(S^{3}_{\theta}, \Sigma) \approx 
\left( 
C^\infty(S^{3}, \Sigma)\widehat\otimes C^\infty(\Tb^2_{\theta})
\right)^{ \alpha\otimes \beta^{-1}}.
\end{equation}
The result is canonically a topological bimodule over $C^\infty(S^{3}_{\theta})$.
Using the usual trace functional $ \tau$ on $C^\infty(\Tb^2_{\theta})$
and given a hermitian structure on $\Sigma$
there is also a hermitian structure 
$$
   (\psi\otimes t, \psi'\otimes t')= (\psi,\psi') \tau (t^\ast t') 
$$
with respect to which $C^\infty(S^{3}_{\theta})$ can be completed to a Hilbert space 
$\H_\theta$.
Moreover, any (also densely defined) operator $D$ which commutes with the action 
$\alpha$ 
defines an operator
 $$D_{\theta} := (D \otimes I) \downharpoonright C^\infty(S^{3}_{\theta}, \Sigma).$$

Now  $S^{3}$ is a Riemannian spin manifold
and the action $\alpha$ of $\Tb^2$ is given by the
left and/or right multiplications by 
particular diagonal elements ${\rm diag}(u, u^*) $ of $SU(2)$, where we identify 
$S^3\equiv SU(2)$.
Since $SU(2) \equiv Spin(3)$ is just a subgroup of $Spin(4)$ which in turn
is the lift to spinors of the connected isometry group $SO(4)$ of   $S^3$,
both the $U(1)$-factors, and so $\Tb^2$ itself lift directly (not just projectively)
to the spinor bundle $\Sigma$ and to its smooth sections called Dirac spinors.
Next, if the $\alpha$-invariant $D$ is the Dirac operator on $\Sigma$,
a spectral triple $(C^\infty(S^{3}_{\theta}),\H_{\theta}, D_{\theta})$
has been defined in \cite{CD-V}.
Here
$$
\H_\theta := \left(L^2(S^{3},\Sigma) \widehat\otimes L^2( \Tb^2_{\theta})\right)
^{\alpha\otimes\beta^{-1}}
$$
and  $D_{\theta}$ is the closure of $D\otimes I$. 
Similarly the 
antilinear charge conjugation operator $J$ can be theta-deformed as
$$J_{\theta}=J\otimes \ast .$$
As shown in \cite{CD-V} the spectral triple $(C^\infty(S^{3}_{\theta}),\H_{\theta}, 
D_{\theta})$ 
together with the real structure $J_{\theta}$ 
satisfies all additional seven axioms \cite{C95,C96} required for a 'noncommutative manifold'.\\

It is also not difficult to verify that the theta deformation
(i.e. the twisting construction as above), behaves 'functorially' under the maps between 
manifolds
(in particular under bundle projection) and respects the properties of principal
$U(1)$-bundles\footnote{for simplicity we assume that the $U(1)$-action coincides 
with the action of the second factor of  $\Tb^2$ above as well as our requirements 
for projectable spectral triples and compatible connections.
Then also the requirements (i-iii) in section \ref{tDo} can be seen to be satisfied.}.
We illustrate this observations by the example of the $\theta$-deformation
of the $U(1)$-Hopf fibration $S^3\to S^2$.

\subsection{$\theta$-deformation of the $U(1)$-Hopf fibration $S^3\to S^2$}


The $U(1)$ action on Matsumoto algebra $\CA(S^3_\theta)$ provides an interesting 
example of a noncommutative Hopf fibration with the commutative  invariant subalgebra
(identified with the algebra of functions on the sphere $S^2$).  
We shall explicitly describe here in a concrete Hilbert space basis 
the Dirac operator on $\CA(S^3_\theta)$ arising from the construction in \cite{cl01}
and will show that it is consistent with the base-space equivariant Dirac operator 
and the strong connection of the deformed magnetic monopole \cite{BrzSi}.
We use the description of the algebra of $\CA(S^3_\theta)$ as generated
by normal operators $a,b$ and their hermitian conjugates, which satisfy
the following relations:
\begin{equation}
ab = \lambda ba, \;\; ab^* = \bar{\lambda} b^*a, \;\; aa^* + bb^* =1,
\end{equation}
where $\lambda = e^{2\pi i \theta}$.

We begin with the representation of the algebras acting on the Hilbert space 
of square integrable function over $S^3$. Using the standard orthonormal basis, 
we have the explicit formulae:
\begin{equation}
\begin{aligned}
\pi_0(a) \ket{l,m,n} &=& \lambda^{-n} \left(
\frac{\sqrt{l+1+m}\sqrt{l+n+1}}{\sqrt{2l+1}\sqrt{2l+2}}
\ket{l^+, m^+, n^-} \right. \\
&& \left. - \frac{\sqrt{l-m}\sqrt{l-n}}{\sqrt{2l}\sqrt{2l+1}}
\ket{l^-,m^+,n^-}\right) ,
\end{aligned}
\end{equation}
\begin{equation}
\begin{aligned}
\pi_0(b) \ket{l,m,n} &=& \lambda^{n} \left(
\frac{\sqrt{l+1+m}\sqrt{l-n+1}}{\sqrt{2l+1}\sqrt{2l+2}}
\ket{l^+, m^-, n^-} \right.\\
&& \left. + \frac{\sqrt{l-m}\sqrt{l+n}}{\sqrt{2l}\sqrt{2l+1}}
\ket{l^-,m^-,n^-}\right) ,
\end{aligned}
\end{equation}
where $l^\pm,m^\pm,n^\pm$ is a shortcut notation for
$l \pm \oh, m\pm\oh,n\pm\oh$. Here $l=0,\oh,1,\ldots$ and both $m,n$ are in 
$-l, l-1+1,\ldots,l-1,l$.
To pass to a spinor representation we need only to double the Hilbert space, so
the spinor representation of $S^3_\theta$ is diagonal:
\begin{equation}
 \pi(x) = \left( \begin{array}{cc} \pi_0(x) & 0
 \\ 0 & \pi_0(x) \end{array} \right),
\end{equation}
The real structure $J$ is off-diagonal:
\begin{equation}
J = \left( \begin{array}{cc} 0 & J_0^-
\\ J_0^+  & 0 \end{array} \right),
\end{equation}
where $J_0^\pm$ is a canonical equivariant antilinear map, which
maps the algebra to its commutant:
\begin{equation}
J_0^\pm \ket{l,m,n} = \pm i^{2(m+n)} \ket{l,-m,-n}.
\end{equation}
Let us verify that $J^2=-1$:
$$
\begin{aligned}
J^2 \ket{l,m,n,\pm}
&= \pm J_0^\mp \left( i^{2(m+n)} \ket{l,-m,-n,\mp} \right) \\
&= - i^{-2(m+n)} i^{-2(m+n)} \ket{l,m,n,\pm} = - \ket{l,m,n,\pm},
\end{aligned}
$$
where we have used that $m+n$ is always integer.

As the Dirac operator we take the following densely defined operator:
\begin{equation}
\begin{aligned}
D \ket{l,m,n,+} &=  \left( r (m+\oh) + \frac{\alpha}{4} \right) \,\ket{l,m,n,+} \\
& +  s \sqrt{l+1+m} \sqrt{l-m} \, \ket{l,m+1,n,-}, \\
D \ket{l,m,n,-} &= \left(- r (m-\oh) + \frac{\alpha}{4} \right) \, \ket{l,m,n,-} \\
& +  s^* \sqrt{l-m+1} \sqrt{l+m} \, \ket{l,m-1,n,+},
\end{aligned}
\label{Diracs3}
\end{equation}
where $s$ is arbitrary complex number and $r$ and $\alpha$ are real. 
It is the most general Dirac-type operator, which is equivariant under the
slightly reduced symmetry of the agebra, which is the product of twisted  
SU(2) and of U(1). For $r=s=\alpha=1$ this operator is isospectral to the 
fully equivariant Dirac operator over $S^3$ (but with the radius of the 
sphere r=$\frac{1}{2}$). 

It is easy to verify that this operator satisfies the usual conditions 
$D^\dagger = D$ and $JD = DJ$. 
\begin{lemma}
The following $U(1)$ action on the Hilbert space:
$$ \ket{l,m,n,\pm}  \to \phi \cdot \ket{l,m,n,\pm} = 
e^{i (2 m \pm 1) \phi} \ket{l,m,n,\pm}, $$
commutes with the Dirac operator $D$ and generates, by conjugation,
the following action on the algebra:
$$ a \to e^{i\phi} a, \;\;\; b \to e^{-i\phi} b. $$

The corresponding unbounded selfadjoint operator $\delta$ on the Hilbert space $\CH$ 
is:
$$ \delta \ket{l,m,n,\pm} = (2 m \pm 1) \ket{l,m,n,\pm}, $$ 
and the above real spectral triple is $U(1)$ equivariant.
\end{lemma}
\begin{proof}
We compute the commutation of $D\rq{}= D - \frac{\alpha}{4} \id$ with the action 
of  $U(1)$:
$$ 
\begin{aligned}
& \phi \cdot (D' \ket{l,m,n\pm} = \phi \cdot \left(
\pm r (m \pm \oh) \,\ket{l,m,n,\pm} 
    + s^{\pm} \sqrt{l+1\pm m} \sqrt{l \mp m} \, \ket{l,m\pm 1,n, \mp} \right) \\
&= \pm r (m \pm \oh) \, e^{i(2m \pm 1)} \ket{l,m,n,\pm} 
         + s^{\pm} \sqrt{l+1\pm m} \sqrt{l \mp m} \, e^{i(2(m \pm 1) \mp 1)} 
\ket{l,m \pm 1,n,\mp} \\
&=  e^{i(2m \pm 1)} \left( \pm r (m \pm \oh) \,  \ket{l,m,n,\pm} 
         + s^{\pm} \sqrt{l+1\pm m} \sqrt{l \mp m} \, \ket{l,m \pm 1,n,\mp} \right) \\
&= D' (\phi \cdot \ket{l,m,n,\pm} ).        
\end{aligned}
$$
Since $D$ differs from $D'$ only by a multiple
of the unit operator, it satisfies the same identity.
\end{proof}

The invariant subalgebra $\CA^{U(1)}$ is generated by $B=b a, B^*=a^* b^*$ and 
$A = aa^*-\oh$, with the relations:
$$ AB= BA, \;\;\; AB^* = B^* A, \;\;\; A^2 + BB^* = \frac{1}{4}, $$
and therefore we shall identify it with the algebra of functions on the 
two-dimensional sphere, $\CA(S^2)$.
\begin{lemma}
There exists an operator $\Gamma$, $\Gamma^2 =1$ and $\Gamma J = -J \Gamma$
such that the spectral triple over $\CA(S^3_\theta)$ is projectable.
\end{lemma}
\begin{proof}
If we take:
$$\Gamma \ket{l,m,n,\pm} = \pm \ket{l,m,n,\pm}, $$
then it is easy to see that all conditions for $\Gamma$ are satisfied. 
\end{proof}

The horizontal Dirac operator is:
\begin{eqnarray*}
D_h \ket{l,m,n, +} &= s   \sqrt{l+1+m} \sqrt{l-m} \, \ket{l,m+1,n,-} \\
D_h \ket{l,m,n, -} &= s^* \sqrt{l-m+1} \sqrt{l+m} \, \ket{l,m-1,n,+}
\end{eqnarray*}
and we have:
\begin{lemma}
The noncommutative $U(1)$ principal bundle has fibers of constant length $\frac{2}{r}$
\end{lemma}
\begin{proof}
 Indeed, a simple computation shows the identity valid on the dense subset 
of the Hilbert space $\CH$:
\begin{equation}
 D = D_h + \frac{r}{2} \Gamma \delta + Z, \label{S3split}
\end{equation}
where $Z = \frac{\alpha}{4} \id$.
\end{proof}

Next, let us look at the invariant subspace of the Hilbert space $\CH$, by construction
it shall provide us with a real spectral triple over the invariant subalgebra.

\begin{lemma}
The spectral triple obtained by restriction of $D$, $J$, and of the operators $A,B$ 
to the invariant subspace of $\CH$ is unitarily equivalent with the standard equivariant 
spectral triple over $S^2$.
\end{lemma}
\begin{proof}
The invariant subspace of $\CH$, $\CH_0 \subset \CH$ is a closure of the linear span
of the following vectors: $\ket{l,-\oh,n,+}$ and $\ket{l,\oh,n,-}$ for 
$l=\oh,\frac{3}{2},\ldots$. Identifying these vectors with the standard equivariant 
basis of the Hilbert space of the spinor bundle over $S^2$ we obtain the desired 
isometry. Again, a simple check verifies that $\Gamma$ and $J$ are indeed the same
as in the case of standard triple over $S^2$. The restriction of $D$ to this space is:
$$ D \ket{l,-\oh,n,+} = s   (l+\oh) \, \ket{l,\oh,n,-}, \;\;\;
D_h \ket{l,\oh, n, -} = s^* (l+\oh) \, \ket{l,-\oh,n,+}, $$
so indeed we recover the standard Dirac operator.
\end{proof}
\begin{lemma}
The first order differential calculus generated by $D$ is compatible with the usual
de Rham differential calculus on $C^\infty(S^1)$.
\end{lemma}
\begin{proof}
Let us observe that using the split of $D$ (\ref{S3split}) we have:
$$ [D, x] = [D_h, x] + \frac{r}{2} \Gamma \delta(x). $$
Since $x$ commutes with $\Gamma$ and $D_h$ anticommutes 
with $\Gamma$, then computing the anticommutator of
$ \sum_i p_i [D, q_i] $, for any $p_i,q_i \in \CA(S^3_\theta)$ we
obtain:
$$ r \sum_i p_i \delta(q_i). $$
Therefore if the first expression vanishes so must the latter and therefore
the calculus generated by $D$ is compatible with the de Rham calculus over
the circle.
\end{proof}

We pass now to the strong connection.
\begin{thm}
The following one-form is a strong connection for the spectral triples 
over the noncommutative $U(1)$-bundle $ \CA(S^2) \hookrightarrow \CA(S^3_\theta) $:
\begin{equation}
\omega = \frac{r}{2} (a^* [D,a] +b [D,b^*]). \label{strongs3}
\end{equation}
and the Dirac operator $D$ (\ref{Diracs3}) is compatible with $\omega$ in the sense
of Definition \ref{defcomp}. 
\end{thm}
\begin{proof}
From construction it is clear that $\omega$ satisfies both $U(1)$ invariance:
$$ [\delta, \omega] = 0, $$
as well as the vertical field condition (modified for a suitable length of fibers):
$$ a^* \delta(a) + b \delta(b^*) = a^* a + b b^* =1,$$
so the only nontrivial condition to verify is strongness. 

To see that it is satisfied first we observe that the bimodule structure of 
$\Omega^1_D(\CA(S^3_\theta)$ resembles the commutation rules of the
algebra, that is $a \, da = da \, a$, $a \, db = \lambda db \, a$ 
etc. As a consequence $da,da^*,db,db^*$ are the generators of 
$\Omega^1_D(\CA(S^3_\theta)$ 
both as a left and as a right module. Moreover, the one forms $dA, dB, dB^*$ 
are central elements of this bimodule.

To demonstrate strongness, we need to show that $[D,x] - \delta(x) \omega 
\in \Omega^1(\CA(S^2)) \CA(S^3_\theta)$. First, let us take $x$ to be the generators
$a$ and $b$. Then, it could be explicitly verified that:

$$
\begin{aligned}
& da - \delta(a) \omega = a \, dA + b^* \, dB, \\
& db - \delta(b) \omega = \lambda a^* \, dB - b\, dA,
\end{aligned}
$$
and by conjugating both  identities we obtain similar ones for $da^*$ and $db^*$. Then, observe that the operator:
$$ x \mapsto dx - \delta(x) \omega, $$
satisfies a Leibniz rule for any $x,y \in \CA(S^3_\theta)$:
$$ d(xy) - \delta(xy) \omega = (dx - \delta(x) \omega) y + x (dy - \delta(y) \omega), $$
because $\omega$ is a central one-form.
 
Putting it all together, for any polynomial in the generators $a,a^*,b,b^*$ we could show
the strongness, As both the commutator with $D$ as well as $\delta$ extends easily to
smooth functions on $S^3_\theta$, so does the property of strongness.

Next,  by straightforward computation we verify that the strong connection $\omega$ 
(\ref{strongs3}) is in fact equal to $ \frac{r}{2} \Gamma$. Then
using the split of the Dirac (\ref{S3split}) and the construction of $D_\omega$ 
(following Proposition 5.2) the last claim of the theorem follows.
\end{proof}



\end{document}